\documentclass[11pt,letterpaper]{amsart}
\usepackage{amsmath,amsthm,amsfonts,amscd,amssymb,eucal,latexsym,mathrsfs,stmaryrd,enumitem,bm,mathtools,microtype}

\newtheorem{theorem}{Theorem}[section]

\theoremstyle{definition}

\newtheorem{example}[theorem]{Example}

\newcounter{theoremintro}
\newtheorem{theoremi}[theoremintro]{Theorem}

\newcommand{\cZ}{{\mathcal Z}}
\newcommand{\sL}{{\mathscr L}}
\newcommand{\fA}{{\sf A}}
\newcommand{\fS}{{\sf S}}
\newcommand{\Cb}{{\mathbb C}}
\newcommand{\Zb}{{\mathbb Z}}
\newcommand{\Nb}{{\mathbb N}}
\newcommand{\Rb}{{\mathbb R}}
\newcommand{\tr}{\tau}
\newcommand{\eps}{\varepsilon}

\numberwithin{equation}{section}

\DeclareMathOperator{\Sym}{Sym}
\DeclareMathOperator{\supp}{supp}
\DeclareMathOperator{\trs}{tr}
\DeclareMathOperator{\ev}{ev}

\DeclarePairedDelimiter\ceil{\lceil}{\rceil}
\DeclarePairedDelimiter\floor{\lfloor}{\rfloor}

\allowdisplaybreaks

\begin{document}

\title{McDuff factors from amenable actions and dynamical alternating groups}

\author{David Kerr}
\address{David Kerr,
Mathematisches Institut,
University of M{\"u}nster, 
Einsteinstr.\ 62, 
48149 M{\"u}nster, Germany}
\email{kerrd@uni-muenster.de}

\author{Spyridon Petrakos}
\address{Spyridon Petrakos,
Mathematisches Institut,
University of M{\"u}nster, 
Einsteinstr.\ 62, 
48149 M{\"u}nster, Germany}
\email{spetrako@uni-muenster.de}

\date{November 14, 2023}

\begin{abstract}
Given a topologically free action of a countably infinite amenable group on
the Cantor set, we prove that, for every subgroup $G$ of the topological full group 
containing the alternating group, the group von Neumann algebra $\sL G$ is a McDuff factor.
This yields the first examples of nonamenable simple finitely generated groups $G$ for which $\sL G$ is McDuff.
Using the same construction we show moreover that if a faithful action $G\curvearrowright X$ of a countable group
on a countable set with no finite orbits 
is amenable then the crossed product of the associated shift action over a given II$_1$ factor is a McDuff factor. 
In particular, if $H$ is a nontrivial countable ICC group
and $G\curvearrowright X$ is a faithful amenable action of a countable ICC group on a countable set
with no finite orbits, then the group von Neumann algebra 
of the generalized wreath product $H\wr_X G$ is a McDuff factor.
Our technique can also be applied to show
that if $H$ is a nontrivial countable group and $G\curvearrowright X$ is an amenable action
of a countable group on a countable set with no finite orbits then the 
generalized wreath product $H\wr_X G$ is Jones--Schmidt stable.
\end{abstract}

\maketitle

\section{Introduction}

In operator algebra theory central sequences have long played a significant role 
in addressing problems in and around amenability, having been used both as a mechanism 
for producing various examples beyond the amenable horizon
and as a point of leverage for teasing out the finer structure of amenable operator algebras themselves.
In the early 1940s Murray and von Neumann exhibited (sticking to the separable realm, as we do henceforth)
the first example of a II$_1$ factor nonisomorphic to the hyperfinite
II$_1$ factor $R$ by showing that the free group factor $\sL F_2$, unlike $R$, 
does not possess nontrivial central sequences, i.e., does not have what they called {\it property Gamma} \cite{MurvNe43}.
In the late 1960s McDuff employed central sequences and an iterated group-theoretic construction
to engineer an uncountable infinity of pairwise nonisomorphic II$_1$ factors \cite{McD69}. Shortly thereafter she
gave a characterization of II$_1$ factors admitting a pair of central sequences that asymptotically 
noncommute as those which tensorially absorb $R$, i.e., those that have the {\it McDuff property} \cite{McD70}.
A bit later in the 1970s Connes put property Gamma to work in the proof of his theorem
that injectivity implies hyperfiniteness, a cornerstone in the classification of injective 
von Neumann algebras \cite{Con76}.
On the topological side,
central sequences (both in operator and tracial norms) have proven their utility many times over in the 
corresponding Elliott classification program for simple separable nuclear 
C$^*$-algebras, starting in the 1990s and with increasing intensity over the last decade.
For instance, versions of property Gamma and the McDuff property formulated in terms of the uniform trace norm
were critical ingredients in recent work on the Toms--Winter conjecture 
\cite{CasEviTikWhiWin21,CasEviTikWhi22}, one outcome of which
was the equivalence of finite nuclear dimension and $\cZ$-stability 
(tensorial absorption of the Jiang--Su algebra) for nonelementary simple separable unital nuclear C$^*$-algebras.
This equivalence permitted one to install the relatively tractable property of $\cZ$-stability as the 
operative regularity hypothesis in the final classification theorem \cite{GonLinNiu20,EllGonLinNiu15,TikWhiWin17}
and cemented its position as the C$^*$-algebraic analogue of being McDuff.

Many ICC groups will give rise to II$_1$ factors with property Gamma or the McDuff property on account
of asymptotic commutativity relations within the group itself, which can be arranged by taking 
products and/or suitable inductive limit constructions (the ICC property---which asks that the conjugacy class of 
every nontrivial element be infinite---guarantees factoriality, and indeed is equivalent to it by a result of Murray and von Neumann). 
Coming up with examples of simple finitely generated groups
that yield such II$_1$ factors is more difficult. The problem of identifying when an infinite group is simple and finitely generated 
can itself be a delicate task, but there is at least one rich source of examples coming from dynamics, namely the 
alternating groups $\fA (\Gamma ,X)$ of minimal subshift actions $\Gamma\curvearrowright X$ 
of countably infinite groups on the Cantor set \cite{Nek17}. 
These are subgroups of the topological full group (i.e., the group of homeomorphisms locally implemented 
by elements of the acting group) that, in the case of many acting groups $\Gamma$ including $\Zb$,
are known to coincide with the commutator subgroup. 
Juschenko and Monod proved that the topological full group
of a minimal $\Zb$-action on the Cantor set is always amenable, which, by passing to the commutator subgroup
and specializing to subshift actions, gave the first examples of amenable infinite simple finitely generated groups \cite{JusMon13}.
When $\Gamma$ is not virtually cyclic, however, the alternating group of a minimal subshift action
can fail to be amenable \cite{EleMon13,Szo21,KerTuc23}. 
Nevertheless, the first author and Tucker-Drob showed that if $\Gamma$ is amenable
and the action is topologically free then for every subgroup of the topological full containing the alternating group 
the group von Neumann algebra (which is always a II$_1$ factor in this case) has property Gamma \cite{KerTuc23}.
The goal of the present paper is to strengthen this last conclusion to the McDuff property:

\begin{theoremi}\label{T-McDuff}
Let $\Gamma\curvearrowright X$ be a topologically free continuous action of a countably 
infinite amenable discrete group on
the Cantor set, and let $G$ be a subgroup of the topological full group $[[\Gamma\curvearrowright X]]$
containing the alternating group $\fA (\Gamma , X)$. Then the von Neumann algebra $\sL G$
is a McDuff II$_1$ factor.
\end{theoremi}

Applying the above result to the free minimal expansive actions constructed in \cite{EleMon13,Szo21}
we obtain the first examples of nonamenable simple finitely generated groups whose von Neumann algebra
is a McDuff factor. The topologically free minimal expansive actions constructed in Section~8 of \cite{KerTuc23}
give us moreover uncountably many pairwise nonisomorphic such groups.

We will actually show something a little more general (see Theorem~\ref{T-McDuff stronger}). The virtue of formulating
Theorem~\ref{T-McDuff} as we have is that the groups $G$ in question are automatically ICC.

The argument in \cite{KerTuc23} for deriving property Gamma makes use  of 
finite permutational wreath products inside of $\fA (\Gamma , X)$
that can be expressed spectrally as permutational Bernoulli actions $S_F \curvearrowright \{ 0,1 \}^F$
indexed by F{\o}lner sets $F$ of $\Gamma$.
A set of measure one half is constructed in each of the Bernoulli spaces $\{ 0,1 \}^F$
via a summation condition on the coordinates that takes into account the F{\o}lner boundary effect.
This set is shown to be approximately invariant using the central limit theorem, as was done by
Kechris and Tsankov in \cite{KecTsa08} for the different purpose of obtaining a characterization 
of amenability for actions in terms of the existence of approximately invariant sets of measure one half
for the associated generalized Bernoulli actions.
The corresponding projection $e$ in the group von Neumann algebra 
is then approximately central to within a prescribed tolerance, yielding property Gamma.
A natural strategy for boosting this to the McDuff property would be 
to take the tensor factors in the group algebra of the wreath product
to be something noncommutative instead of the algebra $\Cb^2$ sitting over the original Bernoulli base $\{ 0,1 \}$,
reinterpreting the original binary alternative as a choice of a seed projection $p$ in these (common) tensor factors,
and then choosing a second seed projection $q$ that is far from commuting with $p$ and 
using it in the same way as $p$ to construct another almost central projection $f$.
We have been unable to determine, however, if such seed projections $p$ and $q$ can be found
so that the corresponding $e$ and $f$ asymptotically noncommute as the F{\o}lner sets $F$ become more and more invariant.
In fact we suspect, on the basis of numerical computations carried out for us by Giles Gardam, that such $e$ and $f$
will always asymptotically commute, even when $p$ and $q$ are approximately freely related.

What we have discovered is that one can dispense with the above probabilistic approach altogether
and instead start with a projection $p$ in the noncommutative Bernoulli base
which has trace extremely close to $1$, close enough so that if we copy it out into a single elementary tensor 
over the F{\o}lner core of the set $F$ then we will obtain a projection $\tilde{p}$ with trace approximately one half. This requires
that the Bernoulli base be a finite-dimensional $*$-subalgebra $B$ of $\sL G$ of very large dimension. 
Some basic representation theory for finite alternating groups (guaranteeing that $B$ can be chosen with enough
noncommutativity) then enables us to construct a partial
isometry $v$ in $B$ such that if we copy it into an elementary tensor $\tilde{v}$ over the F{\o}lner core of $F$,
just like we did to $p$ in order to produce $\tilde{p}$, then we will have $\tilde{v}^* \tilde{v} = \tilde{p}$
and the projections $\tilde{v} \tilde{v}^*$ and $\tilde{p}$ will commute and be approximately independent,
which implies that the commutator $[\tilde{v},\tilde{v}^*]$ is bounded away from zero in trace norm.
Finally one observes that both $\tilde{v}$ and $\tilde{v}^*$ are approximately central, with the tolerance being controlled
by the approximate invariance of the set $F$. From this we conclude the McDuff property. 

Our construction can also be applied to establish a connection to amenability for actions
in the spirit of Kechris and Tsankov, only now via the McDuff property  
for the crossed product of shift actions over a II$_1$ factor. In this setting there is also an obstruction 
related to inner amenability of the group that prevents one from obtaining a full characterization of amenability for the action.
Recall that a group action $G\curvearrowright X$
on a set is {\it amenable} if there is a finitely additive $G$-invariant probability measure on $X$,
or equivalently if there is a state on $\ell^\infty (X)$ for the induced $G$-action.
Every action of an amenable group is amenable, but many nonamenable groups
admit nontrivial and even faithful amenable actions (see \cite{vDo90} for the case of 
free groups and \cite{Tuc20} for further discussion and references).
A group $G$ is {\it inner amenable} if there exists a finitely additive atomless 
probability measure on $G\setminus \{ e \}$ which is invariant under the conjugation action of $G$,
which in the case that $G$ is ICC simply means that the conjugation action 
$G\curvearrowright G\setminus \{ e \}$ is amenable.
Inner amenability fails for free groups on two or more generators 
but does hold for many nonamenable groups. 
It is implied by property Gamma \cite{Eff75} (so that the groups $G$ in Theorem~\ref{T-McDuff} all satisfy it,
as was already shown in \cite{KerTuc23})
but is strictly weaker \cite{Vae12}.

If the action $G\curvearrowright X$ has a finite orbit then it is amenable for obvious reasons and the factoriality condition on the
crossed product below fails, and so this case is naturally omitted from the theorem statement.

\begin{theoremi}\label{T-McDuff shift}
Let $M$ be a II$_1$ factor with trace $\tr$. Let $G\curvearrowright X$ 
be an amenable action of a countable group on a countable set,
and suppose that the action has no finite orbits. Suppose furthermore that the crossed product $M^{\overline{\otimes} X} \rtimes G$
of the associated shift action 
$G\curvearrowright M^{\overline{\otimes} X}$ is a II$_1$ factor (which will be the case, for example, when 
the action $G\curvearrowright X$ is faithful). Then $M^{\overline{\otimes} X} \rtimes G$ is McDuff.
\end{theoremi}

In the special case when the action is that of an amenable group on itself by left translation,
the conclusion of the theorem follows from a general result of Bédos on actions
of amenable groups on McDuff II$_1$ factors \cite{Bed90}. 

When $G$ is non-inner-amenable, the amenability of the action $G\curvearrowright X$
is actually equivalent to both the McDuff property and property Gamma 
for $M^{\overline{\otimes} X} \rtimes G$, with the implication from property Gamma to amenability
following from \cite{Eff75} and Lemma~2.7 of \cite{VaeVer15}, 
as observed in Proposition~2.8 of \cite{Pat23} (the non-inner-amenability assumption
cannot be dropped here, as we illustrate in Example~\ref{free}).
Moreover, Patchell has recently shown, using deformation/rigidity techniques, that if 
$G$ is non-inner-amenable and ICC, the action
$G\curvearrowright X$ is transitive and nonamenable, and the stabilizer of
some nonempty finite subset of $X$ is amenable (a kind of mixing condition)
then $M^{\overline{\otimes} X} \rtimes G$ is prime \cite{Pat23}. 

As a special case of Theorem~\ref{T-McDuff shift}, we obtain the following result 
for II$_1$ factors arising as group von Neumann algebras of generalized wreath products
that conform, as crossed products, to the framework of the theorem statement. 
Recall that the generalized (restricted) wreath product 
$H\wr_X G$ of two groups relative to an action $G\curvearrowright X$ on a set is defined as the semidirect product
$H^{\oplus X} \rtimes G$ where $G$ acts on the restricted direct sum 
$H^{\oplus X}$ by $g\cdot (h_x )_{x\in X} = (h_{g^{-1}x} )_{x\in X}$. In this case there is a natural
isomorphism $\sL (H\wr_X G) \cong \sL (H)^{\overline{\otimes} X} \rtimes G$, and under this identification
we get the two factoriality conditions in Theorem~\ref{T-McDuff shift} precisely when both $H$ and $H\wr_X G$ are ICC.

\begin{theoremi}\label{T-McDuff wreath}
Let $H$ be a nontrivial countable ICC group and $G\curvearrowright X$ be an amenable action of a 
countable group on a countable set with no finite orbits 
such that the generalized wreath product $H\wr_X G$ is ICC (which will be the case, for example, if $G$ is ICC). 
Then $\sL (H\wr_X G)$ is a McDuff II$_1$ factor.
\end{theoremi}

As before, when $G$ is non-inner-amenable the amenability of the action $G\curvearrowright X$
is equivalent to both the McDuff property and property Gamma for $\sL (H\wr_X G)$.

It was shown in \cite{IoaPopVae13,BerVae14} that many generalized wreath products are W$^*$-superrigid, 
i.e., uniquely determined as groups by their group von Neumann algebra.
The base groups in \cite{IoaPopVae13,BerVae14} are Abelian, in contrast to the above
ICC hypothesis on $H$, which is there to guarantee the factoriality of $\sL H$ and hence the applicability
of Theorem~\ref{T-McDuff shift}. The recent papers \cite{ChiIoaOsiSun23a,ChiIoaOsiSun23b} however
treat wreath-like products that include ones with ICC bases.

In response to a question of Robin Tucker-Drob, we show that our technique can also be used
to establish the following result on JS-stability for generalized wreath products.
A countable discrete p.m.p.\ (probability-measure-preserving) equivalence 
relation is said to be {\it JS-stable}
if it satisfies the McDuff-like property of being isomorphic to its product 
with the unique ergodic hyperfinite p.m.p.\ equivalence relation \cite{JonSch87}.
A countable group is {\it JS-stable} if it admits a free ergodic p.m.p.\ action 
whose orbit equivalence relation is JS-stable. We thank Robin Tucker-Drob
for a suggestion that permitted us to remove the non-Abelianness assumption on $H$
in our original version of the theorem.

\begin{theoremi}\label{T-stable}
Let $H$ be a nontrivial countable group and $G\curvearrowright X$ an amenable action of a countable
group on a countable set with no finite orbits. Then the generalized wreath product $H\wr_X G$ is JS-stable.
\end{theoremi}

The details of the proof of Theorem~\ref{T-McDuff}, along with a review of definitions and basic background material, 
are contained in Section~\ref{S-proof}. The proofs of Theorem~\ref{T-McDuff shift} and Theorem~\ref{T-stable}
are contained in Section~\ref{S-proof shift} and \ref{S-proof stable}.
\smallskip

\noindent{\it Acknowledgements.}
The authors were supported by the Deutsche Forschungsgemeinschaft
(DFG, German Research Foundation) under Germany's Excellence Strategy EXC 2044-390685587,
Mathematics M{\"u}nster: Dynamics--Geometry--Structure, and by the SFB 1442 of the DFG.
We are grateful to Robin Tucker-Drob for asking whether our technique could be applied
to the problem of JS-stability for generalized wreath products as well as for discussions on
this topic, to Gregory Patchell for helpful correspondence, and to Giles Gardam for implementing some 
enlightening numerical computations during the early stages of the project.

\section{Proof of Theorem~\ref{T-McDuff}}\label{S-proof}

Let $M$ be a II$_1$ factor and $\tr$ its faithful normal tracial state. Our factors are always
assumed to be separable
for the trace norm $\| a \|_2 = \tr (a^* a)^{1/2}$. A bounded sequence $(a_n )$ in $M$ is said to be
a {\it central sequence} if $\| [a_n , b] \|_2 \to 0$ for every $b\in M$.
The factor $M$ has the {\it McDuff property} if $M\overline{\otimes} R \cong M$ where $R$ is the hyperfinite II$_1$ factor. 
By a theorem of McDuff \cite{McD70}, $M$ has the McDuff property if and only if
there exist central sequences $(a_n )$ and $(b_n )$ in $M$ such that $\| [a_n , b_n ] \|_2 \not\to 0$.
It is this central sequence criterion that we will use to establish the McDuff property in Theorem~\ref{T-McDuff}. Later in Section~\ref{S-proof shift} we will also invoke {\it property Gamma},
which asks for the existence of a central sequence of unitaries with trace zero, a condition
that is readily seen to be weaker than the McDuff property (which itself can also be characterized
by the existence of a sequence of unital $2\times 2$ matrix subalgebras that is central
in the obvious sense).

Let $\Gamma$ be a countable discrete group and $\Gamma\curvearrowright X$ a continuous action
on the Cantor set. The {\it topological full group} $[[ \Gamma\curvearrowright X ]]$ of the action is the group
of all homeomorphisms $h$ from $X$ to itself for which there exist a clopen partition $\{ A_1 ,\dots , A_n \}$
of $X$ and $s_1 , \dots , s_n \in \Gamma$ such that $h(x) = s_i x$ for all $i=1,\dots , n$ and $x\in A_i$. 
This group is countable because $\Gamma$ is countable and $X$ admits only countably many clopen partitions.

Next we recall from \cite{Nek17} the definition of the alternating group $\fA (\Gamma ,X)$.
Let $d\in\Nb$ and write $\fS_d$ for the symmetric group on $\{ 1,\dots , d \}$. 
Consider the homomorphisms $\psi : \fS_d \to [[ \Gamma\curvearrowright X ]]$
for which there exist pairwise disjoint sets $A_1 , \dots A_d \subseteq X$  
such that the image of a permutation $\sigma$ under $\psi$ acts as the identity
on the complement of $A_1 \sqcup\cdots\sqcup A_d$ and for each $i=1,\dots ,d$ maps $A_i$ to $A_{\sigma (i)}$ 
via some element of $\Gamma$. The image of these homomorphisms generate a subgroup $\fS_d (\Gamma ,X)$ of
$[[ \Gamma\curvearrowright X ]]$, and we can also consider the subgroup $\fA_d (\Gamma ,X)$ of $\fS_d (\Gamma ,X)$
generated by the images of the restrictions of the homomorphisms to the alternating group $\fA_d \subseteq \fS_d$.
The group $\fA_3 (\Gamma ,X)$ is called the {\it alternating group} of the action $\Gamma\curvearrowright X$
and written $\fA (\Gamma ,X)$. When the actions has no finite orbits one has $\fA (\Gamma ,X) = \fA_3 (\Gamma ,X)$
for every $d\geq 2$.
It is shown in \cite{Nek17} that if the action of $\Gamma$ is minimal then $\fA_d (\Gamma ,X)$ is simple,
while if $\Gamma$ is finitely generated and the action is expansive 
(equivalently, is a subshift action with finitely many symbols) and has no orbits of cardinality less than $5$
then $\fA_d (\Gamma ,X)$ is finitely generated.

We invariably denote by $\tr$ the unique normal tracial state on a II$_1$ factor, 
with the particular algebra being understood from the context. The identity element of a group will always 
be written $e$.

By Proposition~5.1 of \cite{KerTuc23}, if $\Gamma\curvearrowright X$ is a topologically free continuous action 
of group on the Cantor set then every subgroup of $[[\Gamma\curvearrowright X ]]$ 
containing $\fA (\Gamma, X)$ is ICC. Theorem~\ref{T-McDuff} is thus a consequence of the following result. The amenability of the group $\Gamma$ will be applied in the form 
of the F{\o}lner property, which requires, 
for every finite set $e\in K\subseteq \Gamma$ and $\delta > 0$, that there 
exist a nonempty finite set $T\subseteq \Gamma$ such that $|\bigcap_{s\in K} s^{-1} T | \geq (1 - \delta )|T|$.

\begin{theorem}\label{T-McDuff stronger}
Let $\Gamma\curvearrowright X$ be an action of a countably infinite amenable group on the Cantor set with at least one free orbit.
Then the group von Neumann algebra of any ICC subgroup of $[[\Gamma\curvearrowright X ]]$ 
containing $\fA (\Gamma, X)$ is a McDuff II$_1$ factor.
\end{theorem}

\begin{proof}
Factoriality follows from the ICC condition.

Let $\Omega$ be a finite symmetric subset of $[[\Gamma\curvearrowright X]]$ and let $\eps>0$. By the definition of the topological full group, we can find a clopen partition $\mathcal{P}=\{ P_1 , \dots , P_N \}$ of $X$ such that for each $h\in\Omega$ there exist $s_{h,1} , \dots , s_{h,N} \in \Gamma$ for which $hx = s_{h,i} x$ for every $i=1,\dots , N$ and $x\in P_i$. Let $K$ be the collection of all of these $s_{h,i}$ together with the identity element of $\Gamma$.

Take a $\delta > 0$ such that
\begin{gather}\label{E-delta}
4^{-\delta /(1-\delta )} \geq 1 - \frac{\eps}{32} .
\end{gather}
By taking a logarithm and applying l'H\^{o}pital's rule, one can verify, for all $\theta\in\Rb$, that
\begin{gather}\label{E-log}
\lim_{r\to\infty} (2\cdot 2^{-r^{-1}} - 1)^{\theta r} = \frac{1}{4^\theta} .
\end{gather}
It follows that we can find an $r_0 > 0$ so that, for all $r\geq r_0$ and $\theta \in \{ 1 , \delta /(1-\delta ) \}$,
\begin{gather}\label{E-theta}
\Big|(2\cdot 2^{-r^{-1}} - 1)^{\theta r} - \frac{1}{4^\theta} \Big| \leq \frac{\eps}{32} .
\end{gather}

By hypothesis there exists an $x_0 \in X$ such that
the action of $\Gamma$ on the orbit $\Gamma x_0$ is free. 
By amenability, there exists a nonempty finite set $T\subseteq\Gamma$ such that the set 
$T'=\bigcap_{s\in K} s^{-1} T$ (which is a subset of $T$ since $e\in K$) satisfies 
$|T' | \geq \ceil{(1-\delta )|T|}$.
Since $\Gamma$ is infinite, we can choose $T$ so that its cardinality is larger that $r_0$, and also large 
enough so that
\begin{gather}\label{E-ceil}
(2^{-((1-\delta)|T|)^{-1}} )^{\ceil{(1-\delta)|T|}} \geq \frac12 - \frac{\eps}{8} .
\end{gather}
Since the action of $\Gamma$ 
on the orbit of $x_0$ is free and $\Gamma$ is infinite, we can find a sequence $(d_k)$ in $\Gamma$ such that the points $td_kx_0$ for $k\in\mathbb{N}$ and $t\in T$ are pairwise distinct. By the pigeonhole principle 
we can find a subsequence $(d_{k_j})$ such that for every $t\in T$ the points $td_{k_j}x_0$ for $j\in\mathbb{N}$ belong to a common member of $\mathcal{P}$. 
In particular, using continuity we can find a finite set $D\subset\Gamma$ of cardinality as large as we wish
(to be specified below) and a clopen neighbourhood $B$ of $x_0$ such that the sets $tdB$ for $d\in D$ and $t\in T$ are pairwise disjoint and for every $t\in T$ the sets $tdB$ for $d\in D$ 
are contained in a common member of $\mathcal{P}$. This choice of $D$ guarantees the existence of a function $\theta\in K^{\Omega\times T}$ defined by $htdx=\theta(h,t)tdx$ for $h\in\Omega, t\in T, d\in D$, and $x\in B$.

For each $s\in K$ we have $sT' \subseteq T$ and so for every $h\in\Omega$ we can find a $\sigma_h \in\Sym (T)$ such that $\sigma_ht = \theta (h,t)t$ for all $t\in T'$. 
Consider the alternating group $\fA (D)$. We regard the product $\fA (D)^{T}$ as a subgroup of $\fA(\Gamma, X)$ with an element $(\omega_t)_{t\in T}$ in $\fA (D)^{T}$ acting by $dtx\mapsto\omega_t(d)tx$ for all $x\in B$, $t\in T$, and $d\in D$ and by $x\mapsto x$ for all $x\in X\setminus TDB$.

By the representation theory of alternating groups \cite{FulHar91,JamKer81},
the Artin--Wedderburn decomposition of $\sL\fA(D)$ takes the form 
$\mathbb{C}\oplus\big(\bigoplus_{l\in L}\mathbb{M}_{k_l}\big)$ where $k_l\geq |D| -1$ for every $l$
in the finite set $L$. Write $\trs$ for the tracial state
on $\sL\fA(D)$ associated to the left regular representation of $\fA(D)$ on $\ell^2 (\fA(D))$,
i.e., the vector state $a\mapsto \langle a\delta_e , \delta _e \rangle$ where $\{ \delta_g : g\in\fA(D) \}$
is the canonical orthonormal basis for $\ell^2 (\fA(D))$. 
The summand $\Cb$ in the Artin--Wedderburn decomposition corresponds to the trivial representation 
of $\fA(D)$ and thus must act on $\ell^2 (\fA(D))$ as the orthogonal projection onto the
one-dimensional subspace of $\fA(D)$-invariant vectors, which is spanned by
the unit vector $\xi = |\fA(D)|^{-1/2} \sum_{g\in \fA(D)} \delta_g$. It follows that 
the projection $f := (1,0) \in \mathbb{C}\oplus\big(\bigoplus_{l\in L}\mathbb{M}_{k_l}\big)$ satisfies
\begin{gather}\label{E-trivial}
\trs (f) = \langle f\delta_e , \delta_e \rangle
= \langle \langle \delta_e , \xi \rangle \xi , \delta_e \rangle
= \frac{1}{|\fA(D)|} .
\end{gather}
Consequently there exist $\lambda_l > 0$ for $l\in L$ with $|\fA(D)|^{-1} + \sum_{l\in L} \lambda_l = 1$ 
(in fact $\lambda_l = k_l^2 / (1+\sum_{l\in L} k_l^2 )$ by standard theory) 
such that for all $a = (b , (c_l )_{l\in L} ) \in \mathbb{C}\oplus\big(\bigoplus_{l\in L}\mathbb{M}_{k_l}\big)$
we have, denoting by $\trs_l$ the unique tracial state on $M_{k_l}$,
\begin{gather}\label{E-convex}
\trs (a) = \frac{b}{|\fA(D)|} + \sum_{l\in L} \lambda_l \trs_l (c_l ) .
\end{gather}

For $l\in L$ write $\{ e_{i,j}^{(l)} \}_{1\leq i,j\leq k_l}$ for the matrix units of the summand $M_{k_l}$.
Set $d_l =\floor{2^{-((1-\delta)|T|)^{-1}}k_l}$ and define
\begin{align*}
    p&=\sum_{l\in L}\sum_{i=1}^{d_l}e_{i,i}^{(l)},\\
    q&=\sum_{l\in L}\Bigg(\sum_{i=1}^{2d_l-k_l}e_{i,i}^{(l)}+\sum_{i=d_l+1}^{k_l}e_{i,i}^{(l)}\Bigg),\\
    v&=\sum_{l\in L}\Bigg(\sum_{i=1}^{2d_l-k_l}e_{i,i}^{(l)}+\sum_{i=d_l+1}^{k_l}e_{i,i-k_l+d_l}^{(l)}\Bigg).
\end{align*}
We have $v^*v=p$, $vv^*=q$, and $vpq=pq$. Moreover, writing $e_l = \sum_{i=1}^{d_l} e_{i,i}$ we have, 
by (\ref{E-convex}),
\begin{align*}
\trs (p) = \trs (q)
= \sum_{l\in L} \trs (e_l ) 
= \sum_{l\in L} \lambda_l \frac{d_l}{k_l} .
\end{align*}
It follows, by virtue of the equation $\sum_{l\in L} \lambda_l = 1 - |\fA(D)|^{-1}$, 
our choice of $d_l$, and the fact that $k_l \geq |D|-1$ for every $l\in L$,
that we can make the quantity $\trs (p)$ as close to $2^{-((1-\delta)|T|)^{-1}}$ as we wish by taking $|D|$ sufficiently large.
Thus given an $\eps' > 0$ we can take $D$ to have large enough cardinality so that
\begin{gather*}
\trs (p) \geq 2^{-((1-\delta)|T|)^{-1}} - \eps' .
\end{gather*}
Since $\trs (pq) = \sum_{l\in L} \lambda_l (2d_l - k_l ) /k_l$,
we may similarly assume that $|D|$ is large enough so that
\begin{gather*}
2\cdot 2^{-((1-\delta)|T|)^{-1}} -1 - \eps' \leq \trs (pq) \leq 2\cdot 2^{-((1-\delta)|T|)^{-1}} -1 .
\end{gather*}
Therefore by taking $\eps'$ small enough we can guarantee, in view of (\ref{E-ceil}), that
\begin{gather}\label{E-one}
\trs (p)^{\ceil{(1-\delta)|T|}} \geq (2^{-((1-\delta)|T|)^{-1}} - \eps' )^{\ceil{(1-\delta)|T|}} \geq \frac12 - \frac{\eps}{4} 
\end{gather}
and, in view of (\ref{E-theta}), that
\begin{gather}\label{E-two}
\trs (pq)^{\delta |T|} \geq (2\cdot 2^{-((1-\delta)|T|)^{-1}} - 1- \eps' )^{\delta |T|} \geq 1 - \frac{\eps}{8} .
\end{gather}
Note also by (\ref{E-theta}) that 
\begin{gather}\label{E-three}
\trs (pq)^{(1-\delta)|T|} \leq (2\cdot 2^{-((1-\delta)|T|)^{-1}} - 1)^{(1-\delta )|T|} \leq \frac14 + \frac{\eps}{4} .
\end{gather}

For every $a\in\sL\fA(D)$ and $R\subseteq T$ write $\tilde{a}_R$ for the element 
$\otimes_{t\in T}a_t\in\sL\fA(D)^{\otimes T}=\sL (\fA (D)^{T} )\subseteq\sL\fA(\Gamma,X)$ where $a_t=a$ if $t\in R$ and $a_t=1$ otherwise. Note that the restriction of the trace $\tr$ on $\sL\fA(\Gamma,X)$ to $\sL (\fA (D)^{T} )$,
under the identification of the latter with $\sL\fA(D)^{\otimes T}$, is equal to the tensor product trace $\trs^{\otimes T}$.
Note also that for $h\in\Omega$, $a\in\sL\fA(D)$, and $S\subseteq T'$ we have $u_h\tilde{a}_Su_h^{-1}=\tilde{a}_{\sigma_hS}$. 
Since $T'$ has cardinality at least $\ceil{(1-\delta)|T|}$, we can choose a $S\subseteq T'$ with exactly
this cardinality. 
It follows using (\ref{E-one}) and (\ref{E-three}) that
\begin{align*}
    \|[\tilde{v}_{S},\tilde{v}^*_{S}]\|_2^2&=\|\tilde{p}_{S}-\tilde{q}_{S}\|_2^2\\
    &=\tr(\tilde{p}_{S}+\tilde{q}_{S}-2\widetilde{pq}_{S})\\
    &=2\big(\trs (p)^{|S|}-\trs (pq)^{|S|}\big)\\
    &\geq\frac{1}{2}-\eps.
\end{align*}
Furthermore, for all $h\in\Omega$ we have $|\sigma_hS\setminus S| \leq \delta |T|$
and hence, using (\ref{E-two}),
\begin{align}\label{E-commutator}
    \| u_h\tilde{v}_{S}u_h^{-1}-\tilde{v}_{S}\|_2^2&=\|\tilde{v}_{\sigma_hS}-\tilde{v}_{S} \|_2^2\\
    &=\|\tilde{v}_{\sigma_hS\cap S} (\tilde{v}_{\sigma_hS\setminus S}-\tilde{v}_{S\setminus\sigma_hS}) \|_2^2 \notag\\
    &\leq \| \tilde{v}_{\sigma_hS\cap S} \|^2\|\tilde{v}_{\sigma_hS\setminus S}-\tilde{v}_{S\setminus\sigma_hS} \|_2^2 \notag\\
    &\leq (\|\tilde{v}_{\sigma_hS\setminus S}-1 \|_2 + \| 1- \tilde{v}_{S\setminus\sigma_hS} \|_2 )^2 \notag\\
    &= 4\|\tilde{v}_{\sigma_hS\setminus S}-1 \|_2^2 \notag\\
    &= 4\tr (\tilde{v}_{\sigma_hS\setminus S}^* \tilde{v}_{\sigma_hS\setminus S} -
    \tilde{v}_{\sigma_hS\setminus S}^*- \tilde{v}_{\sigma_hS\setminus S} +1) \notag\\
    &\leq 8 \tr (1-\tilde{v}_{\sigma_hS\setminus S} ) \notag\\
    &= 8(1-\trs (pq)^{|\sigma_h S \setminus S|} ) \notag\\
    &\leq 8(1-\trs (pq)^{\delta |T|} ) \notag\\
    &\leq\eps .\notag
\end{align}
Thus if we take such $\tilde{v}_{S}$ and $\tilde{v}^*_{S}$ over an increasing sequence of sets $\Omega$ 
with union $[[\Gamma\curvearrowright X]]$ and a sequence of tolerances $\eps$ converging to zero, we 
obtain noncommuting approximately central sequences for $\sL[[\Gamma\curvearrowright X]]$ 
inside of $\sL\fA(\Gamma,X)$. This yields the McDuff conclusion in the theorem.
\end{proof}

In the above proof we could have avoided the application of l'Hospital's rule
by alternatively taking the seed projections $p$ and $q$ to be approximately independent with respect to the trace.
In fact this is how we will proceed in the proof of Theorem~\ref{T-McDuff shift},
where the whole picture simplifies due to the diffuseness of the seed space.

\section{Proof of Theorem~\ref{T-McDuff shift}}\label{S-proof shift}

It is a standard fact (provable in the same way as for groups acting on themselves by translation) that amenability for a group action $G\curvearrowright X$ on a set
is equivalent to the following F{\o}lner property: 
for every finite set $e\in K\subseteq G$ and $\delta > 0$ there
exists a nonempty finite set $T\subseteq X$ such that $|\bigcap_{s\in K} s^{-1} T | \geq (1 - \delta )|T|$,
in which case we say that $T$ is {\it $(K,\delta )$-invariant}. Such sets $T$ are informally referred to as
{\it F{\o}lner sets} with the understanding that a certain degree of approximate invariance is at play.
It can also be shown, in the say way as for groups themselves with respect to the left regular representation,  
that amenability for an action $G\curvearrowright X$ is equivalent to the existence
of approximately invariant unit vectors for the induced unitary representation $\pi$ on $\ell^2 (X)$,
i.e., to the existence, for every finite set $F\subseteq G$ and $\delta > 0$, of a unit vector $\xi\in\ell^2 (X)$
satisfying $\| \pi (g) \xi - \xi \| < \delta$ for every $g\in F$.
If $G$ itself is amenable (i.e., the action of $G$ on itself by left translation is amenable) 
then all of its actions are amenable, as is easy to verify.
See Section~4.1 of \cite{KerLi16} for more information.

\begin{proof}[Proof of Theorem \ref{T-McDuff shift}]
We wish to show, given a finite subset $\Omega$ of $M^{\overline{\otimes}X}\rtimes G$ and an $\eps>0$, 
that there exists a pair of elements in $M^{\overline{\otimes}X}\rtimes G$ whose commutators with elements in $\Omega$
have trace norm less than $\eps$ and whose commutator with each other has trace norm 
bounded away from zero independently of $\eps$. 
It evidently suffices to check this for $\Omega$ drawn from a subset of the crossed product which generates 
a trace-norm dense subalgebra. We may thus assume that
$\Omega=\Omega_1\cup\Omega_2$ where $\Omega_1$ consists of elementary tensors in $M^{\overline{\otimes}X}$ 
of finite support and $\Omega_2$ is the set $\{ u_g \}_{g\in F}$ of canonical unitaries corresponding to elements in a 
given finite subset $F$ of $G$ containing $e$. Write $Y$ for the union of the supports of elements in $\Omega_1$.

Choose a $\delta>0$ small enough so that $(1-2^{-2\delta /(1-\delta )}) \leq \eps /8$.
Since by assumption the action on $X$ has no finite orbits, 
the cardinality of the F{\o}lner sets for the action will tend to infinity as we demand more and more invariance.
We can thus find an $(F,\delta)$-invariant subset $T$ of $X$ that is disjoint from $Y$ by
first shrinking the tolerance $\delta$
a little and then finding a F{\o}lner set for this tightened tolerance that has sufficiently large cardinality 
so that its intersection with the complement of $Y$ will do the job.
Set $S=\bigcap_{s\in F}s^{-1}T$, which by $(F,\delta)$-invariance satisfies 
$|S| \geq (1-\delta )|T|$. Since $M$ is a II$_1$ factor it contains commuting projections $p$ and $q$ of trace 
$2^{-|S|^{-1}}$ which are independent, i.e., $\tr (pq) = \tr (p)\tr (q)$ 
(for example, choose a masa in $M$, write it in the form $A\overline{\otimes} A$ 
in such a way that the trace $\tr$ on $M$ restricts to
$\tr |_{A\otimes 1} \otimes \tr |_{1\otimes A}$ under the canonical identification of the two copies of $A$ 
with $A\otimes 1$ and $1\otimes A$, and take $f\otimes 1$ and $1\otimes f$ for a suitable projection $f$).
Since the projections $p-pq$ and $q-pq$ have the same trace they are Murray--von Neumann equivalent,
and so we can construct a partial isometry $v\in M$ such that $v^*v=p$, $vv^*=q$, and $vpq=pq$. 
Note that 
\begin{align}\label{E-pitrace}
	\tr (v) = \tr (qvp) = \tr (vpq) = \tr (pq) .
\end{align}

For $a\in M$ and $R\subseteq T$ write $\tilde{a}_R$ for the element $\otimes_{t\in T}a_t \in M^{\overline{\otimes}X}$ 
where $a_t=a$ if $t\in R$ and $a_t=1$ otherwise.
Then $\tr (\tilde{p}_S ) = \tr (\tilde{q}_S ) = \tr (p)^{|S|} = 1/2$ and $\tr (\widetilde{pq}_S ) = \tr (p)^{2|S|} = 1/4$.
By our choice of $T$, both $\tilde{v}$ and $\tilde{v}^*$ commute with the elements in $\Omega_1$. 
Moreover
\[
	\|[\tilde{v}_{S},\tilde{v}^*_{S}]\|_2^2
	=\|\tilde{p}_{S}-\tilde{q}_{S}\|_2^2
	=\tr(\tilde{p}_{S} ) + \tr (\tilde{q}_{S} )-2\tr (\widetilde{pq}_{S})
	= \frac12 ,
\]
and, using (\ref{E-pitrace}) and estimating as in (\ref{E-commutator}) in the proof of Theorem \ref{T-McDuff stronger},
we have, for every $g\in F$,
\begin{align*}
	\| u_g^{-1}\tilde{v}_{S}u_g-\tilde{v}_{S}\|_2^2&=\|\tilde{v}_{gS}-\tilde{v}_{S} \|_2^2\\
	&\leq 4\|\tilde{v}_{gS\setminus S}-1 \|_2^2 \\
	&= 4\tr (\tilde{v}_{gS\setminus S}^* \tilde{v}_{gS\setminus S} -
	\tilde{v}_{gS\setminus S}^*- \tilde{v}_{gS\setminus S} +1) \\
	&\leq 8 \tr (1-\tilde{v}_{gS\setminus S} ) \\
	&= 8(1-\tr (pq)^{|gS \setminus S|} ) \\
	&\leq 8(1-(2^{-2|S|^{-1}} )^{\delta |T|} ) \\
	&\leq 8(1-2^{-2\delta /(1-\delta )}) \\
	&\leq\eps .
\end{align*}
Taking such $\tilde{v}_S$ and $\tilde{v}_S^*$
over an increasing sequence of finite sets $\Omega$ 
with trace-norm dense union in $M^{\overline{\otimes}X}\rtimes G$ and a sequence of tolerances $\eps$ converging to zero,
we obtain central sequences witnessing the McDuff property.
\end{proof}

As mentioned in the introduction, it follows by \cite{Eff75} and Lemma~2.7 of \cite{VaeVer15}
that if $G$ in the context of 
Theorem~\ref{T-McDuff shift} is assumed to be non-inner-amenable then 
the action $G\curvearrowright X$ is amenable whenever $M^{\overline{\otimes} X} \rtimes G$
has property Gamma, and so in this case amenability
of the action $G\curvearrowright X$ is equivalent to both the McDuff property and property Gamma 
for $M^{\overline{\otimes} X} \rtimes G$. The non-inner-amenability cannot be dropped here,
as the following example illustrates.

\begin{example}\label{free}
Let $G:=F_2\times F_2^{\oplus\Nb}$ act on the set $X:=F_2\times\Nb$ by $(s,(t_k)_{k\in\Nb} )\cdot (r,n)=(srt_n^{-1},n)$. 
This action is not amenable, for if there were a finitely additive $G$-invariant probability measure 
on $X$ then it would give nonzero measure to $F_2\times \{ n \}$ for some $n$, and since the 
action of the subgroup $F_2 \times \{ e \} \subseteq G$ on $F_2\times \{ n \}$ simply replicates the action of $F_2$
on itself by left translation this would yield a contradiction to the nonamenability of $F_2$.

Now let $M$ be any II$_1$ factor, let $\Omega_1\subseteq M^{\overline{\otimes} X}$ be a finite set of elementary tensors 
with finite support, and let $F$ be a finite subset of $G$. Write $\Omega_2$ for the set $\{ u_g \}_{g\in F}$ of unitaries 
in the crossed product $M^{\overline{\otimes} X}\rtimes G$ corresponding to $F$.
For an element $(s,(t_k)_{k\in\Nb} )$ in $G$ we will call the (finite) set of indices $k$ in $\Nb$ for which $t_k\neq e$ 
its \emph{support}. Furthermore, for $s\in F_2$ and $n\in\Nb$ we denote by $g_{s,n}$ 
the element $(e,(t_k)_{k\in\Nb})$ supported on $\{n\}$ with $t_n=s$. Write $K_1$ for the 
set of all $n\in\Nb$ such that $F_2\times \{ n \}$ intersects the support of some element of 
$\Omega_1$ and write $K_2$ for the union of the supports of the elements in $F$. 
Then $K:= K_1 \cup K_2$ is a finite subset of $\Nb$. Write $a,b$ for the generators of $F_2$. 
Then for every $n\in\Nb\setminus K$ and $s\in\{a,b\}$ we have $gg_{s,n}g^{-1}=g_{s,n}$ for all $g\in F$ 
and $u_{g_{s,n}}yu_{g_{s,n}}^* = y$ for all $y\in\Omega_1$. Furthermore, 
$\| [u_{g_{a,n}},u_{g_{b,n}}] \|_2=\sqrt{2}$, and so one can construct a noncommuting sequence of such pairs 
of unitaries which are asymptotically central, showing that the II$_1$ factor
$M^{\overline{\otimes} X}\rtimes G$ is McDuff.
\end{example}

\section{Generalized wreath products and JS-stability}\label{S-proof stable}

Recall that the full group of a p.m.p.\ action $G\curvearrowright (Y,\nu )$ is the set of all
measurable maps $T:Y\rightarrow G$ with the property that the transformation 
$T^0$ of $Y$ given by $T^0(y):= T(y)y$ is a measure automorphism.
By Kida's general version of a criterion due to Jones and Schmidt in the free ergodic case \cite{JonSch87,Kid15},
to verify that a countable group $G$ is JS-stable it suffices to show that it admits 
a p.m.p.\ action $G\curvearrowright (Y,\nu )$ possessing a {\it stability sequence}, i.e.,
a sequence of pairs $(T_n,A_n)$ where $T_n$ is a member of the full group and $A_n$ is
a measurable subset of $X$ such that
\begin{enumerate}
	\item $\nu (\{ y\in Y : T_n (gy) =gT_n(y)g^{-1} \} ) \to 1$ for every $g\in G$,
	\item $\nu (T_n^0(B)\Delta B ) \to 0$ for every measurable $B\subseteq Y$,
	\item $\nu (gA_n\Delta A_n)\to 0$ for every $g\in G$,
	\item $\nu (T_n^0(A_n)\Delta A_n ) \geq \frac12$ for all $n\in\Nb$.
\end{enumerate}

The following proof uses the same kind of idea as in Sections~\ref{S-proof} and \ref{S-proof shift},
but there is an additional technical twist here in the 
construction of the full group elements in the definition of stability sequence, one that has no
analogue in von Neumann algebra framework of the previous two sections.

\begin{proof}[Proof of Theorem \ref{T-stable}]
Denoting by $\lambda$ the Lebesgue measure on $[0,1]$, we consider 
the p.m.p. action $H\wr_X G\overset{\alpha}{\curvearrowright}(Y,\nu):=( ([0,1]^H)^X,(\lambda^H)^X)
= ( [0,1]^{H\times X},\lambda^{H\times X})$ induced by the given action $G\curvearrowright X$
and the Bernoulli action $H\overset{\gamma}{\curvearrowright} ([0,1]^H , \lambda^H )$, as determined by
$\alpha_{ag} ((y_{s,x} )_{s\in H, x\in X}) = (y_{a_x^{-1} s,g^{-1}x} )_{s\in H, x\in X}$ 
for $a = (a_x )_{x\in X} \in H^{\oplus X}$ and $g\in G$.
By the discussion above, it suffices to show that $\alpha$ admits a stability sequence.
	
To that end, let $F=\{\tilde{h}_ig_i : i\in I\}$ be a finite subset of $H\wr_X G$ with $\tilde{h}_i\in H^{\oplus X}$
and $g_i\in G$ for every $i\in I$. Set $W=\bigcup_{i\in I}\supp\tilde{h}_i$ and $K=\{g_i:i\in I\}$, and take 
an $\eps>0$ such that $2^{-3\eps}>1-|F|^{-1}$. As in the proof of Theorem~\ref{T-McDuff shift}, 
we can find a finite subset $E$ of $X$ that is $(K,\eps)$-invariant and disjoint from $W\cup\bigcup_{i\in I} g_i^{-1}W$. 
	
For $s\in H$ write $\pi_s :[0,1]^H\to[0,1]$ for the projection map onto the coordinate at $s$, and set $U_0=[0,2^{-|E|^{-1}}]$
and $U_1=[0,1]\setminus U_0$. Let $h\in H\setminus\{e\}$ and consider the map $\omega :[0,1]^H\to H$ that is equal to $h$ on $Z:=\pi_e^{-1}(U_0)\cap\pi_h^{-1}(U_1)$, to $h^{-1}$ on the image of $Z$ under the shift $\gamma_h$
(which is disjoint from $Z$), and to $e$ otherwise.
Define $T:Y\to H^{\oplus{X}}\subseteq H\wr_XG$ by declaring,
for all $y\in Y$ and $x\in X$, that
\[
	T(y)(x)=
	\begin{cases}\omega (y_x), & x\in E,\\
	e, & \text{otherwise}.
	\end{cases}
\]
By construction, $T$ is an element of the full group of $\alpha$. Set 
\[
A=\big\{ (y_x )_{x\in X} \in Y:y_x \in \pi^{-1}_e (U_0) \text{ for all } x\in E \big\} .
\]
	
Because $E$ is disjoint from $W$, both $T$ and its image are invariant under the action (via $\alpha$ and by conjugation, respectively) of any element of $H^{\oplus X}$ supported on $W$. Let $g\in K$, denote by $Y_g$ the set of 
all $y\in Y$ such that $T(\alpha_g(y))=gT(y)g^{-1}$, and set 
\[
	C=\big\{ (y_x )_{x\in X} \in Y:y_x \in\pi^{-1}_e(U_0)\cap\pi^{-1}_h(U_0)\cap\pi^{-1}_{h^2}(U_0)
	\text{ for all } x\in E\Delta g^{-1}E\big\} .
\]
One can easily check that $C\subseteq Y_g$, and therefore
\[
	\nu(Y_g)\geq\nu(C)\geq(2^{-|E|^{-1}})^{3\eps|E|}\geq1-\frac{1}{|F|}.
\]
Moreover, by construction we have $T_0(B)=B$ for all measurable rectangles $B$ such that $\ev_x(B)=[0,1]^H$ for all $x\in E$, where $\ev_x$ denotes the evaluation map at $x$.
		
Finally, we have
\[
	\nu(T_0(A)\Delta A)=2(\nu(A)-\nu(T_0(A)\cap A))=2\Big(\frac{1}{2}-\frac{1}{4}\Big)=\frac{1}{2}
\]
and, for $t\in F$,
\begin{align*}
	\nu(\alpha_{t}(A)\Delta A)&=2(\nu(A)-\nu(\alpha_{t}(A)\cap A))\leq2\Big(\frac{1}{2}-\Big(\frac{1}{2}\Big)^{(1+\eps)}\Big)\leq\frac{1}{|F|}.
\end{align*}
	
By constructing such $A$ and $T$ with respect to an increasing sequence of finite sets $F_n$ such that 
$\bigcup_nF_n= H\wr_X G$, we obtain a stability sequence for $\alpha$.
\end{proof}

\end{document}